\documentclass[11pt,twoside,a4paper,article]{memoir}

\usepackage[latin1]{inputenc}
\usepackage[T1]{fontenc}

\usepackage[english]{babel}

\usepackage{graphicx}

\usepackage{microtype}

\usepackage[justification=RaggedRight]{subfig}

\usepackage{amsmath}
\usepackage{amssymb}
\usepackage{amsfonts}

\usepackage[amsmath,thmmarks,hyperref]{ntheorem}

\usepackage{mathtools}

\usepackage[justification=RaggedRight]{subfig}

\usepackage[sc]{mathpazo}

\usepackage{url}

\usepackage[ps,all]{xy}
\SelectTips{cm}{}
\CompileMatrices

\makeatletter
\@ifpackageloaded{mathpazo}{}{%
  \usepackage{fix-cm}
  \usepackage{bbm}}
\makeatother

\ifpdf
\usepackage[pdftitle={The first cohomology of the mapping class group
  with coefficients in algebraic functions on the SL(2, C) moduli
  space},
pdfauthor={Jørgen Ellegaard Andersen, Rasmus Villemoes},
pdfkeywords={mapping class group, multicurve, group cohomology},
colorlinks=true, 
pdftex, backref=none]{hyperref}
\else
\usepackage[pdftitle={The first cohomology of the mapping class group
  with coefficients in algebraic functions on the SL(2, C) moduli
  space},
pdfauthor={Jørgen Ellegaard Andersen, Rasmus Villemoes},
pdfkeywords={mapping class group, multicurve, group cohomology},
colorlinks=true, 
ps2pdf,
bookmarksnumbered=true,
backref=none]{hyperref}
\fi
\usepackage{memhfixc}

\makeatletter
\@ifpackageloaded{mathpazo}{
  \newcommand{\mathsetfont}{\mathbb}
}{
  \newcommand{\mathsetfont}{\mathbbm}}
\makeatother

\newcommand{\DeclareMathSet}[1]{%
  \expandafter\newcommand\csname set#1\endcsname{\mathsetfont{#1}}}

\DeclareMathSet{C} 
\DeclareMathSet{R} 
\DeclareMathSet{Q} 
\DeclareMathSet{Z} 
\DeclareMathSet{N} 

\DeclareMathOperator{\Map}{Map}
\DeclareMathOperator{\Hom}{Hom}

\DeclareMathOperator{\Coind}{Coind}       

\def\cxymatrix#1{\xy*[c]\xybox{\xymatrix#1}\endxy}

\newcommand{\inv}{^{-1}}

\renewcommand{\phi}{\varphi}
\renewcommand{\epsilon}{\varepsilon}
\renewcommand{\rho}{\varrho}

\newcommand{\SL}{\mathrm{SL}}
\newcommand{\SU}{\mathrm{SU}}

\theoremstyle{plain}
\theoremsymbol{}
\theorembodyfont{\itshape}
\theoremheaderfont{\normalfont\bfseries}
\theoremseparator{.}
\newtheorem{lemma}{Lemma}
\newtheorem{theorem}[lemma]{Theorem}
\newtheorem{proposition}[lemma]{Proposition}

\theoremstyle{nonumberplain}
\theorembodyfont{\normalfont}
\theoremheaderfont{\itshape}
\theoremsymbol{\ensuremath{\square}}
\newtheorem{proof}{Proof}

\qedsymbol{\ensuremath{\square}}

\setsecnumdepth{subsection}

\makepagestyle{rv}
\makeoddhead {rv}{}{\textsc{Mapping class groups and algebraic
    functions}}{\thepage}
\makeevenhead{rv}{\thepage}{\textsc{J.E. Andersen, R. Villemoes}}{}
\author{Jørgen Ellegaard Andersen \and Rasmus Villemoes}

\title{The first cohomology of the mapping class group with coefficients in
algebraic functions on the $\SL_2(\setC)$ moduli space}

\begin{document}

\maketitle

\begin{abstract}
  \noindent
  Consider a compact surface of genus at least two. We prove
  that the first cohomology group of the mapping class group with
  coefficients in the space of algebraic functions on the
  $\SL_2(\setC)$ moduli space vanishes.
\end{abstract}

\section{Introduction}
\label{sec:introduction}

Let $\Sigma$ be a compact surface, possibly with boundary, of genus at
least $2$, and let $\mathcal{M} = \mathcal{M}_{\SL_2(\setC)}$ denote
the moduli space of flat $\SL_2(\setC)$ connections over $\Sigma$.
Since $\mathcal{M}$ may be identified with the space of $\SL_2(\setC)$
representations of the fundamental group of $\Sigma$ modulo
conjugation, $\mathcal{M}$ has the structure of an algebraic variety.
The mapping class group $\Gamma$ acts on $\mathcal{M}$ and hence on
the space $\mathcal{O} = \mathcal{O}(\mathcal{M})$ of algebraic
functions on $\mathcal{M}$, making $\mathcal{O}$ a module over
$\Gamma$. The purpose of the present paper is to prove
\begin{theorem}
  \label{thm:1}
  The first cohomology group $H^1(\Gamma, \mathcal{O})$ vanishes.
\end{theorem}

The proof relies crucially on the $\Gamma$-equivariant identification
of $\mathcal{O}$ with another vector space on which the action of
$\Gamma$ is more transparent. Based on Goldman's idea of using curves
in the surface to represent functions on the moduli space, Bullock,
Frohman and Kania-Bartoszy{\'n}ska in \cite{MR1691437} (see also
\cite{ARS2006}) proved that $\mathcal{O}$ is $\Gamma$-equivariantly
isomorphic to the complex vector space spanned by the set of
multicurves on $\Sigma$. This allows one to decompose $\mathcal{O}$
into smaller $\Gamma$ modules indexed by the mapping class group
orbits of multicurves.

In \cite{0710.2203}, we computed the cohomology group $H^1(\Gamma,
\mathcal{O}^*)$, where $\mathcal{O}^*$ denotes the algebraic dual of
$\mathcal{O}$. Using the set of multicurves as a basis for
$\mathcal{O}$, there is an inclusion map $\iota \colon
\mathcal{O}\to\mathcal{O}^*$, and this induces a map on cohomology
$H^1(\Gamma, \mathcal{O})\to H^1(\Gamma, \mathcal{O}^*)$. Using the
description of the target given in \cite{0710.2203}, we first prove
that this map is zero, and then, using the results from
\cite{0802.3000}, we prove that it is injective.

\section{Motivation}
\label{sec:motivation}

The motivation for studying the first cohomology group of the mapping
class group with coefficients in a space of functions on the moduli
space came from \cite{0611126}. In that paper, the first author
studied deformation quantizations, or star products, of the Poisson
algebra of smooth functions on the moduli space $\mathcal{M}_G$ of
flat $G$-connections, where $G=\SU(n)$. The construction uses Toeplitz
operator techniques and produces a family of star products
parametrized by Teichmüller space. In \cite{0611126} the problem of
turning this family into one mapping class group invariant star
product was reduced to a question about the first cohomology group of
the mapping class group with various twisted coefficients.
Specifically, one of the results in \cite{0611126} (Proposition~6) is
that, provided the cohomology group $H^1(\Gamma,
C^\infty(\mathcal{M}_G))$ vanishes, one may find a $\Gamma$-invariant
equivalence between any two equivalent star products.  Since it is
easy to see that the only $\Gamma$-invariant equivalences are the
multiples of the identity, this immediately implies that within each
equivalence class of star products, there is at most one
$\Gamma$-invariant star product.

Theorem~\ref{thm:1} is clearly a step towards verifying the assumption
above in the case of $G=SU(2)$, since the $SU(2)$-moduli space is
included in the $\SL_2(\setC)$ moduli space.

\section{Splitting the coefficient module}
\label{sec:splitt-coeff-module}

A \emph{multicurve} is the isotopy class of a finite collection of
pairwise disjoint, simple closed curves on $\Sigma$. Let $B$ denote
the set of multicurves on $\Sigma$, and let $\mathcal{B} =
\mathcal{B}(\Sigma) = \setC B$ denote the complex vector space spanned
by $B$. In \cite{ARS2006} one finds a complete proof of
\begin{theorem}
  \label{thm:2}
  There exists a $\Gamma$-equivariant isomorphism $\nu\colon \mathcal{B}\to
  \mathcal{O}$.
\end{theorem}
If $D = \sqcup_{i=1}^n \gamma_i$ is the disjoint union of simple
closed curves $\gamma_i$, $\nu(D)$ is simply $(-1)^n\prod_{i=1}^n
f_{\vec\gamma_i}$, where $\vec \gamma_i$ denotes any of the oriented
versions of $\gamma_i$, and $f_{\vec\gamma_i}$ is Goldman's holonomy
function on the moduli space.

Theorem~\ref{thm:2} allows us to split $\mathcal{O}$ according to the
mapping class group orbits of multicurves. More precisely, for a
multicurve $D$, let $M_D = \setC(\Gamma D)$ denote the complex vector
space spanned by the $\Gamma$-orbit through $D$. Then we have a
decomposition as $\Gamma$-modules
\begin{align}
  \label{eq:1}
  \mathcal{O} \cong \mathcal{B} \cong \bigoplus_D M_D
\end{align}
where the sum is over a set of representatives of the mapping class
group orbits of multicurves. This induces a corresponding
decomposition of the cohomology
\begin{align}
  \label{eq:2}
  H^1(\Gamma, \mathcal{B}) \cong \bigoplus_D H^1(\Gamma, M_D).
\end{align}
Hence it suffices to show that each summand on the right-hand side of
\eqref{eq:2} vanishes in order to prove Theorem~\ref{thm:1}.

\section{A larger module}
\label{sec:larger-module}

It turns out to ease the computation of $H^1(\Gamma, M_D)$ if one
introduces a larger module. Let $\mathcal{B}^*$ denote the algebraic
dual of $\mathcal{B}$. Using the set of multicurves as a basis, there
is a $\Gamma$-equivariant inclusion $\colon
\mathcal{B}\to\mathcal{B}^*$. In fact, we may identify $\mathcal{B}^*$
with the space $\Map(B, \setC)$ of all formal linear combinations of
multicurves. There is a decomposition of $\mathcal{B}^*$ similar to
\eqref{eq:1} into a direct product of $\Gamma$-modules,
\begin{align}
  \label{eq:3}
  \mathcal{B}^* \cong \prod_D \hat M_D,
\end{align}
where $\hat M_D = \Map(\Gamma D, \setC)$ denotes the set of all formal
linear combinations of elements of the orbit through $D$, and the
product is over the same set of repersentatives as in \eqref{eq:1}.

The $\Gamma$-equivariant inclusion $\iota\colon M_D\to \hat M_D$
induces a long exact sequence in cohomology, the first part of which is
\begin{align}
  \label{eq:4}
  \cxymatrix{{0\ar[r] & H^0(\Gamma, M_D) \ar[r] & H^0(\Gamma, \hat M_D)
      \ar[r] & H^0(\Gamma, \hat M_D/M_D) \\
      \ar[r] & H^1(\Gamma, M_D) \ar[r]^{\iota_*} & H^1(\Gamma, \hat M_D). }}
\end{align}

In \cite{0710.2203}, we computed $H^1(\Gamma, \hat M_D)$ for any
multicurve $D$, and showed that for any surface there exists a
multicurve such that $H^1(\Gamma, \hat M_D)$ is non-zero.

We need the description of $H^1(\Gamma, \hat M_D)$ given in
\cite{0710.2203}, so let us recall the most important facts. Let
$\Gamma_D\subseteq \Gamma$ denote the stabilizer of $D$ in $\Gamma$
(permutation of the components of $D$ are allowed). Then the
$\Gamma$-equivariant identification of the set $\Gamma/\Gamma_D$ of
left cosets with the orbit $\Gamma D$ induces an isomorphism of
$\hat M_D = \Map(\Gamma D, \setC)$ with the space
$\Hom_{\setZ\Gamma_D}(\setZ\Gamma, \setC)$ of
$\setZ\Gamma_D$-homomorphisms $\setZ\Gamma\to\setC$.

This $\Gamma$-module is also known as the co-induced module
$\Coind_{\Gamma_D}^{\Gamma} \setC$, and Shapiro's Lemma (see
\cite{MR672956}) yields an isomorphism
\begin{align}
  \label{eq:5}
  H^1(\Gamma, \hat M_D) = H^1(\Gamma, \Coind_{\Gamma_D}^{\Gamma}
  \setC) \cong H^1(\Gamma_D, \setC)
\end{align}
where $\setC$ is a trivial $\Gamma_D$-module. Hence $H^1(\Gamma, \hat
M_D)$ is simply the space of homomorphisms from (the abelianization
of) $\Gamma_D$ to $\setC$.

Explicitly, the isomorphism \eqref{eq:5} is given as follows: An
element of $H^1(\Gamma, \hat M_D)$ is represented by a cocycle
$u\colon \Gamma\to \hat M_D$, which can also be considered as a map
$u\colon \Gamma\times \Gamma D\to\setC$. Restricting to the subset
$\Gamma_D \times \{D\} \equiv \Gamma_D$ we obtain a map $u_|\colon
\Gamma_D\to \setC$, which is easily seen to be a homomorphism. In
other words, $u_|(g)$ is given by picking out the coefficient of $D$
in $u(g)$.

\section{Dehn twists and multicurves}
\label{sec:dehn-twists-mult}

Before starting actual computations leading to a proof of
Theorem~\ref{thm:1}, we need to record a few facts regarding Dehn
twists, multicurves and the modules $M_D$, $\hat M_D$.

\subsection{Presentations and relations}
\label{sec:pres-relat}

It is well-known that the mapping class group is generated by Dehn
twists. In fact, several finite presentations of $\Gamma$ are known,
where the generators are the twists in a suitable set of simple closed
curves (cf. \cite{MR719117}, \cite{MR1851559}).

For later use, we mention a few relations between Dehn twists.
\begin{lemma}
  \label{lem:2}
  Dehn twists on disjoint curves commute.
\end{lemma}
\begin{lemma}
  \label{lem:3}
  If $\alpha$ and $\beta$ are simple closed curves intersecting
  transversely in a single point, the associated Dehn twists are
  \emph{braided}. That is, $\tau_\alpha\tau_\beta\tau_\alpha =
  \tau_\beta\tau_\alpha\tau_\beta$.
\end{lemma}
\begin{lemma}[Chain relation]
  \label{lem:4}
  Let $\alpha$, $\beta$ and $\gamma$ be simple closed curves in a
  two-holed torus as in Figure~\ref{fig:chain-relation}, and let
  $\delta$, $\epsilon$ denote curves parallel to the boundary
  components of the torus. Then $(\tau_\alpha\tau_\beta\tau_\gamma)^4
  = \tau_\delta\tau_\epsilon$.
\end{lemma}
\begin{figure}[hbt]
  \centering
  \quad
  \subfloat[A two-holed torus.\label{fig:chain-relation-real}]
  {\includegraphics[trim=-10 0 -10 0,clip]{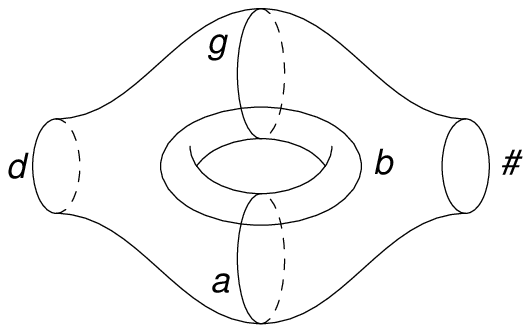}}
  \hfill
  \subfloat[A more schematic picture.\label{fig:chain-relation-schematic}]
  {\includegraphics[trim=-10 0 -10
    0,clip]{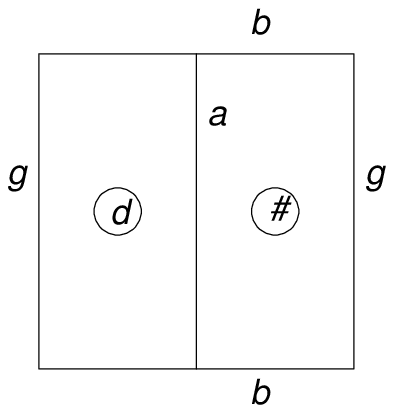}}
  \quad\strut
  \caption{The chain relation.}
  \label{fig:chain-relation}
\end{figure}

\subsection{The action of twists on multicurves}
\label{sec:acti-twists-mult}

There is simple way to parametrize the set of all multicurves which
was found by Dehn. For details, we refer to~\cite{MR1144770}.
Essentially one cuts the surface into pairs of pants using $3g+r-3$
simple closed curves $\gamma_k$, and then for each pants curve
$\gamma_k$ one records the geometric intersection number $m_k(D) =
i(\gamma_k, D)$ (which is a non-negative integer) and a ``twisting
number'' $t_k(D)$, which can be any integer. This defines a
$6g+2r-6$-tuple of integers $(m_1(D), t_1(D), \ldots, m_{3g+r-3}(D),
t_{3g+r-3}(D))$ (satisfying certain conditions), and, conversely, from
any such tuple satisfying these conditions one may construct a
multicurve.

The important fact is that in this parametrization, the action of the
twist in the curve $\gamma_k$ on a multicurve $D$ is given by
\begin{align}
  \label{eq:7}
  t_k(\tau_{\gamma_k}^{\pm 1} D) = t_k(D) \pm m_k(D),
\end{align}
all other coordinates being unchanged. The formula \eqref{eq:7} is
intuitive in the sense that it says that for each time $D$ intersects
$\gamma_k$ essentially, the action of $\tau_{\gamma_k}$ on $D$ adds
$1$ to the twisting number of $D$ with respect to $\gamma_k$. This can
be used to prove a number of important facts.

\begin{lemma}
  \label{lem:5}
  Let $\gamma$ be a simple closed curve and $D$ a multicurve. Then the
  following are equivalent:
  \begin{enumerate}[(1)]\firmlist
  \item The twist $\tau_\gamma$ acts trivially on $D$.
  \item The twist $\tau_\gamma$ acts trivially on each component of
    $D$.
  \item The geometric intersection number between $\gamma$ and $D$ is
    zero.
  \item One may realize $\gamma$ and $D$ disjointly.
  \end{enumerate}
  Conversely, if $\tau_\gamma$ acts non-trivially on $D$, all the
  multicurves $\tau_\gamma^n D$, $n\in\setZ$, are distinct.
\end{lemma}
\begin{proof}
  All of the above assertions can be proved from~\eqref{eq:7} by
  letting $\gamma$ be part of a pants decomposition of the surface.
  This is clearly possible if $\gamma$ is non-separating, while if
  $\gamma$ is separating, observe that both connected components
  resulting from cutting along $\gamma$ must have negative Euler
  characteristic (otherwise $\gamma$ would be trivial or parallel to a
  boundary component, in which case the twist on $\gamma$ clearly acts
  trivially on $D$).
\end{proof}

To find a twist acting non-trivially on a multicurve, we need only
find a curve which has positive geometric intersection number with
the multicurve. This is possible if and only if the multicurve has a
component which is not parallel to a boundary component of $\Sigma$.

\subsection{Isomorphisms of modules}
\label{sec:isomorphisms-modules}

If $D$ is a multicurve, let $D^n$ denote the multicurve obtained from
$D$ by replacing each component by $n$ parallel copies. Clearly, there
are $\Gamma$-isomorphisms $M_D\to M_{D^n}$ and $\hat M_D\to\hat
M_{D^n}$. Also, if $\gamma$ is a simple closed curve parallel to a
boundary component of $\Sigma$, we have $\Gamma$-isomorphisms $M_D\to
M_{D\cup \gamma}$ and $\hat M_D \to \hat M_{D\cup \gamma}$. These
observations imply that we may without loss of generality only
consider multicurves without boundary parallel components, and
satisfying that the multiplicities of the different components are
relatively prime. Using non-standard terminology, such a multicurve
will be called \emph{reduced}.

\section{The map $\iota_*$}
\label{sec:map-iota}

Let $D$ be a reduced multicurve. The purpose of the section is to prove
\begin{proposition}
  \label{prop:1}
  The map $\iota_*\colon H^1(\Gamma, M_D) \to H^1(\Gamma, \hat M_D)$ is
  zero.
\end{proposition}
The proof uses the description of $H^1(\Gamma, \hat M_D)$ as
$\Hom(\Gamma_D, \setC)$ given at the end of
section~\ref{sec:larger-module}. Let $u\colon \Gamma\to M_D$ be a
cocycle. Since $\Gamma$ is generated by Dehn twists, it is natural to
study to which extent $u(\tau_\alpha)$ can contain non-zero terms on
which $\tau_\alpha$ acts trivially for simple closed curves $\alpha$.
\begin{lemma}
  \label{lem:1}
  Let $\alpha$ be a simple closed curve on $\Sigma$, and let $E\in
  \Gamma D$ be a multicurve such that $\tau_\alpha E = E$. Assume that
  $E$ contains at least one component which is not a parallel copy of
  $\alpha$. Then the coefficient of $E$ in $u(\tau_\alpha)$ is zero.
\end{lemma}
\begin{proof}
  Let $\epsilon$ be a component of $E$ which is not parallel to
  $\alpha$. Then since every component of $E$ is disjoint from
  $\alpha$, and since we assumed that $D$ (and hence $E$) is a reduced
  multicurve, $\epsilon$ is not parallel to a boundary component of
  the (possibly disconnected) surface $\Sigma_\alpha$ obtained by
  cutting $\Sigma$ along $\alpha$. Hence we may find a curve $\beta$
  disjoint from $\alpha$ such that $\tau_\beta \epsilon \not=
  \epsilon$ and thus $\tau_\beta E \not= E$. Then $\tau_\alpha$ and
  $\tau_\beta$ commute, and $u(\tau_\alpha \tau_\beta) = u(\tau_\beta
  \tau_\alpha)$. Using the cocycle condition this becomes
  \begin{align*}
    u(\tau_\alpha) + \tau_\alpha u(\tau_\beta) = u(\tau_\beta) +
    \tau_\beta u(\tau_\alpha),
  \end{align*}
  which we may rewrite as
  \begin{align}
    \label{eq:6}
    (1-\tau_\beta)\cdot u(\tau_\alpha) = (1-\tau_\alpha) \cdot u(\tau_\beta).
  \end{align}
  Now since $\tau_\alpha E = E$, the coefficient of $E$ on the
  right-hand side of \eqref{eq:6} is clearly $0$. Assuming that
  $u(\tau_\alpha)$ contains some non-zero term $xE$ then implies that
  it must also contain the term $x\tau_\beta\inv E$. But since
  $\tau_\alpha$ and $\tau_\beta$ commute, $\tau_\alpha$ also acts
  trivially on $\tau_\beta\inv E$, so we may repeat the above argument
  with $\tau_\beta\inv E$ instead of $E$ and conclude that
  $u(\tau_\alpha)$ then also contains the term $x\tau_\beta^{-2}E$.
  Continuing in this way, $u(\tau_\alpha)$ contains infinitely many
  non-zero terms (since the multicurves $\tau_\beta^n E$ are all
  distinct), which is impossible since we assumed that $u$ took values
  in $M_D$.
\end{proof}
In other words, $\tau_\alpha$ acts non-trivially on ``most'' of the
non-zero terms occuring in $u(\tau_\alpha)$; the possible exception is
when $D$ consists of a single component and the curve $\alpha$ is in
the orbit of $D$ (e.g. if $D$ and $\alpha$ are non-separating
curves). But this possibility is easily ruled out.
\begin{proposition}
  \label{prop:2}
  Let $\epsilon$ be any simple closed curve. Then $\tau_\epsilon$ acts
  non-trivially on any non-zero term occuring in $u(\tau_\epsilon)$.
\end{proposition}
\begin{proof}
  By the previous lemma, we only need to prove that $u(\tau_\epsilon)$
  does not contain some non-zero term $x\epsilon$, where $\epsilon$ is
  considered as a $1$-component multicurve. To see this, observe that
  any curve $\epsilon$ can be realized as the $\epsilon$ occuring in
  the chain relation (Lemma~\ref{lem:4}); that is, there exists a
  genus $1$ subsurface of $\Sigma$ with two boundary components, one
  of which is $\epsilon$: If $\epsilon$ is separating, one of the
  connected components obtained by cutting along $\epsilon$ has genus
  $\geq 1$, and we may if necessary choose $\delta$ to be
  null-homotopic. If $\epsilon$ is non-separating, it is always
  possible to find a $\delta$ such that the two curves together bound
  a genus $1$ subsurface.

  Applying the cocycle $u$ to the chain relation, we obtain
  \begin{align*}
    u((\tau_\alpha\tau_\beta\tau_\gamma)^4) = u(\tau_\delta) +
    \tau_\delta u(\tau_\epsilon).
  \end{align*}
  But the left-hand side can be expanded (via the cocycle condition)
  to a sum of various actions of $\tau_\alpha, \tau_\beta,
  \tau_\gamma$ on the values of $u$ on these twists; since they all
  act trivially on $\epsilon$ the coefficient of $\epsilon$ on the
  left-hand side is $0$ by Lemma~\ref{lem:1}. Similarly, $\delta$ acts
  trivially on $\epsilon$, so also the coefficient of $\epsilon$ in
  $u(\tau_\delta)$ is $0$, and hence the coefficient of $\epsilon$ in
  $u(\tau_\epsilon)$ is $0$.
\end{proof}
\begin{proof}[Proposition~\ref{prop:1}]
  Let $u\colon \Gamma\to M_D$ be a cocycle. 
  By the isomorphism \eqref{eq:5} it suffices to prove the following:
  For any diffeomorphism $f\in \Gamma_D$ fixing the multicurve $D$,
  the coefficient of $D$ in $u(f)$ is zero.

  Since $u_|\colon \Gamma_D\to \setC$ is a homomorphism, we may
  consider any power of $f$. Choose $n$ sufficiently large so that
  $f^n$ fixes each component of $D$ and each side of each
  component. Then $f^n$ may be realized as a diffeomorphism of the
  surface $\Sigma_D$ obtained by cutting $\Sigma$ along $D$. This
  implies that $f^n$ can be written as a product of Dehn twists in
  curves not intersecting $D$. Hence, by Proposition~\ref{prop:2}, the
  coefficient of $D$ in $u(f^n)$ is zero, and so is the coefficient of
  $D$ in $u(f)$.
\end{proof}

\section{Proof of the main theorem}
\label{sec:proof-theor}

Before proving Theorem~\ref{thm:1}, we need to quote the main theorem
from \cite{0802.3000}. This requires a little terminology:

Let $\Gamma$ be a group and $X$ an infinite set on which $\Gamma$
acts. We define a \emph{coloring} (or $C$-coloring) of $X$ to be any
map $c\colon X\to C$ into some set $C$ of ``colors''.
\begin{itemize}\tightlist
\item A coloring $c$ is \emph{invariant} if $c(gx) = c(x)$ for
  each $g\in \Gamma$ and $x\in X$.
\item A coloring is \emph{almost invariant} if, for each $g\in \Gamma$, the
  identity $c(x) = c(gx)$ fails for only finitely many $x\in X$.
\item Two colorings are \emph{equivalent} if they assign different colors to
  only finitely many elements of $X$; this is clearly an equivalence
  relation on the set of $C$-colorings.
\item A coloring is \emph{trivial} if it is equivalent to a
  monochromatic (constant) coloring.
\end{itemize}
Letting $\Gamma$ denote the mapping class group of a surface of genus
at least $2$, and $X$ the $\Gamma$-orbit of an arbitrary multicurve,
we have
\begin{theorem}[\cite{0802.3000}]
  \label{thm:3}
  There are no non-trivial almost invariant colorings of $X$.
\end{theorem}

Now we have all the tools we need.
\begin{proof}[Theorem~\ref{thm:1}]
  By the isomorphism \eqref{eq:1} and the splitting \eqref{eq:2}, it
  suffices to prove that each summand $H^1(\Gamma, M_D)$ vanishes. By
  Proposition~\ref{prop:1}, we need only show that the map $\iota_*$
  is injective. By the exact sequence \eqref{eq:4}, this is equivalent
  to proving that $H^0(\Gamma, \hat M_D) \to H^0(\Gamma, \hat
  M_D/M_D)$ is surjective.

  Now, an invariant element of $\hat M_D/M_D$ is represented by an
  element $v\in \hat M_D = \Map(\Gamma D, \setC)$ such that for each
  $g\in \Gamma$ we have $v-gv \in M_D$. Since $(v-gv)(E) = v(E) -
  v(g\inv E)$ for $E\in \Gamma D$, we see that this must be zero for
  all but finitely many $E\in \Gamma D$. In other words, $v$ must be
  an almost invariant $\setC$-coloring of $\Gamma D$ in the above
  language, and since by Theorem~\ref{thm:3} no non-trivial almost
  invariant colorings of $\Gamma D$ exist, we conclude that $v$ is
  almost constant, ie. all but finitely many elements of $\Gamma D$ is
  mapped to the same complex number $z$. But then $v$ represents the
  same element of $\hat M_D/M_D$ as the constant linear combination
  $\sum_{E\in \Gamma D} z E$, and hence $H^0(\Gamma, \hat M_D) \to
  H^0(\Gamma, \hat M_D/M_D)$ is in fact surjective.
\end{proof}

\bibliographystyle{alphaurl}
\bibliography{../phd}

\begin{thebibliography}{BFKB99}

\bibitem[And06]{0611126}
J{\o}rgen~Ellegaard Andersen.
\newblock Hitchin's connection, {T}oeplitz operators and symmetry invariant
  deformation quantization.
\newblock 2006.
\newblock \href {http://arxiv.org/abs/math.DG/0611126}
  {\path{arXiv:math.DG/0611126}}.

\bibitem[AV07]{0710.2203}
J{\o}rgen~Ellegaard Andersen and Rasmus Villemoes.
\newblock Degree one cohomology with twisted coefficients of the mapping class
  group.
\newblock 2007.
\newblock \href {http://arxiv.org/abs/0710.2203} {\path{arXiv:0710.2203}}.

\bibitem[BFKB99]{MR1691437}
Doug Bullock, Charles Frohman, and Joanna Kania-Bartoszy{\'n}ska.
\newblock Understanding the {K}auffman bracket skein module.
\newblock {\em J. Knot Theory Ramifications}, 8(3):265--277, 1999.

\bibitem[Bro82]{MR672956}
Kenneth~S. Brown.
\newblock {\em Cohomology of groups}, volume~87 of {\em Graduate Texts in
  Mathematics}.
\newblock Springer-Verlag, New York, 1982.

\bibitem[Ger01]{MR1851559}
Sylvain Gervais.
\newblock A finite presentation of the mapping class group of a punctured
  surface.
\newblock {\em Topology}, 40(4):703--725, 2001.

\bibitem[PH92]{MR1144770}
R.~C. Penner and J.~L. Harer.
\newblock {\em Combinatorics of train tracks}, volume 125 of {\em Annals of
  Mathematics Studies}.
\newblock Princeton University Press, Princeton, NJ, 1992.

\bibitem[Sko06]{ARS2006}
Anders~Reiter Skovborg.
\newblock {\em The Moduli Space of Flat Connections on a Surface -- Poisson
  Structures and Quantization}.
\newblock PhD thesis, University of Aarhus, 2006.
\newblock \\Available from \url{http://www.imf.au.dk/publs?id=623}.

\bibitem[Vil08]{0802.3000}
Rasmus Villemoes.
\newblock The mapping class group orbit of a multicurve.
\newblock 2008.
\newblock \href {http://arxiv.org/abs/0802.3000} {\path{arXiv:0802.3000}}.

\bibitem[Waj83]{MR719117}
Bronislaw Wajnryb.
\newblock A simple presentation for the mapping class group of an orientable
  surface.
\newblock {\em Israel J. Math.}, 45(2-3):157--174, 1983.

\end{thebibliography}

\end{document}